\documentclass[11pt]{amsart}

\usepackage{amsmath,amsthm,amssymb,amscd}
\usepackage[margin=1in]{geometry}
\usepackage[bookmarks]{hyperref}
\usepackage{verbatim}

\newtheorem{theorem}{Theorem}[subsection]

\newtheorem{corollary}[theorem]{Corollary}
\newtheorem{lemma}[theorem]{Lemma}

\theoremstyle{definition}

\numberwithin{equation}{subsection}
\numberwithin{table}{subsection}

\newcommand{\C}{{\mathbb C}}

\newcommand{\0}{\bar 0}

\newcommand{\1}{\bar 1}

\newcommand{\g}{\ensuremath{\mathfrak{g}}}

\newcommand{\e}{\ensuremath{\mathfrak{e}}}
\newcommand{\f}{\ensuremath{\mathfrak{f}}}

\begin{document}

\title[The Nilpotent Cone for Classical Lie Superalgebras]
{\bf The Nilpotent Cone for Classical Lie Superalgebras}

\author{\sc L. Andrew Jenkins}
\address
{Department of Mathematics\\ University of Georgia \\
Athens\\ GA~30602, USA}
\thanks{Research of the first author was supported in part by NSF (RTG) grant DMS-1344994}
\email{lee.jenkins25@uga.edu}

\author{\sc Daniel K. Nakano}
\address
{Department of Mathematics\\ University of Georgia \\
Athens\\ GA~30602, USA}
\thanks{Research of the second author was supported in part by NSF
grant  DMS-1701768.} 

\email{nakano@uga.edu}

\keywords{Lie superalgebras, representation theory}
\date\today

\begin{abstract} In this paper the authors introduce an analog of the nilpotent cone, ${\mathcal N}$, for a classical Lie superalgebra, ${\mathfrak g}$, that generalizes the 
definition for the nilpotent cone for semisimple Lie algebras. For a classical simple Lie superalgebra, ${\mathfrak g}={\mathfrak g}_{\0}\oplus 
{\mathfrak g}_{\1}$ with $\text{Lie }G_{\0}={\mathfrak g}_{\0}$, it is shown that there are finitely many $G_{\0}$-orbits on ${\mathcal N}$. 
Later the authors prove that the Duflo-Serganova commuting variety, ${\mathcal X}$, is contained in ${\mathcal N}$ for any classical simple 
Lie superalgebra. Consequently, our finiteness result generalizes and extends the work of Duflo-Serganova on the commuting variety. 
Further applications are given at the end of the paper.  
\end{abstract} 

\maketitle

\vskip 1cm

\section{Introduction}

\subsection{} Let ${\mathfrak g}$ be a finite-dimensional Lie algebra over the complex numbers and ${\mathcal N}$ be the set of nilpotent elements, often 
referred to as the {\em nilpotent cone}. In the case when ${\mathfrak g}$ is semisimple with ${\mathfrak g}=\text{Lie } G$, it is well-known that ${\mathcal N}$ has finitely many $G$-orbits. 
For the classical families of simple Lie algebras (root systems of types $A$-$D$) a parametrization of orbits is given by partitions under suitable conditions, and for the exceptional Lie algebras one can either 
use the Bala-Carter labelling or weighted Dynkin diagrams. For a semisimple Lie algebra, ${\mathfrak g}$, fundamental results in geometric representation theory have involved 
investigating the geometry of the nilpotent cone ${\mathcal N}$ (also its Springer resolution $\widetilde{\mathcal N}$) and its relationship to the representation theory for ${\mathfrak g}$ (cf. \cite{HTT}). 

The nilpotent cone ${\mathcal N}$ can be realized as the zero set of the constant term zero $G$-invariant polynomials on ${\mathfrak g}$. The $G$-invariant polynomials also have a direct 
connection with the semisimple elements. Under the Chevalley isomorphism theorem the restriction map induces an isomorphism 
$\text{res}:S({\mathfrak g}^{*})^{G}\rightarrow S({\mathfrak t}^{*})^W$ where ${\mathfrak t}$ is a maximal torus of ${\mathfrak g}$ and $W$ is the Weyl group. 
The semisimple elements are those elements in ${\mathfrak g}$ that are $G$-conjugate to an element in ${\mathfrak t}$ \cite[Section 0.1]{Hum}. 

\subsection{} A similar picture arises in the study of classical simple Lie superalgebras, ${\mathfrak g}={\mathfrak g}_{\0}\oplus {\mathfrak g}_{\1}$. Boe, Kujawa and Nakano \cite{BKN1} 
used invariant theory for reductive groups to show that there are natural classes of ``subalgebras'' that detect the cohomology. These subalgebras arise from 
considering ``semisimple" elements of the $G_{\0}$ action on ${\mathfrak g}_{\1}$, and fall into two families: ${\mathfrak f}$ (when 
$\g$ is stable) and ${\mathfrak e}$ (when $\g$ is polar). If ${\mathfrak g}$ is a classical simple Lie superalgebra, then $\g$ admits a stable action and in most cases $\g$ admits a polar action 
(cf \cite[Table 5]{BKN1}). 

When stable and polar actions exist, the restriction maps induce isomorphisms: 
$$
\CD
\text{H}^{\bullet}(\g,\g_{\0},{\mathbb C}) @>>> \text{H}^{\bullet}(\f,\f_{\0},{\mathbb C})^{N} @>>> \text{H}^{\bullet}(\e,\e_{\0},{\mathbb C})^{W_{\e}} \\
@VVV  @VVV  @VVV\\
S^{\bullet}({\mathfrak g}_{\1}^{*})^{G_{\0}} @>>> S^{\bullet}({\mathfrak f}_{\1}^{*})^{N} @>>> S^{\bullet}({\mathfrak e}_{\1}^{*})^{W_{\e}} \\
\endCD
$$
where $N$ is a reductive group and $W_{\e}$ is a finite pseudoreflection group. 
The finite generation of these cohomology rings was used in \cite{BKN1} to define support varieties for ${\mathfrak g}$-modules. In \cite{GGNW} 
it was shown that the support varieties (appropriately) defined over ${\mathfrak g}$, ${\mathfrak f}$ and ${\mathfrak e}$ are isomorphic. 

In order to have a complete picture, it is natural to ask whether there exists an algebraic variety consisting of ``nilpotent elements" for classical Lie superalgebras that fits into this framework. 
Since ${\mathfrak g}$ is classical, ${\mathfrak g}_{\0}=\text{Lie } G_{\0}$ where $G_{\0}$ is a reductive algebraic group. In this paper we study a generalization of the nilpotent cone 
$${\mathcal N}={\mathcal N}_{\mathfrak g}=Z(S^{\bullet}({\mathfrak g}_{\1}^{*})^{G_{\0}}_{+})$$ 
where $Z(S^{\bullet}({\mathfrak g}_{\1}^{*})^{G_{\0}}_{+})$ is the zero-set of $G_{\0}$-invariant polynomials with constant term zero on ${\mathfrak g}_{\1}$. When ${\mathfrak g}={\mathfrak q}(n)$, 
one obtains the nilpotent cone for the Lie algebra $\mathfrak{gl}_{n}({\mathbb C})$. 

The construction in our paper is inspired by work of Kac \cite{Kac} in 1980. He defined the nilvariety as the zero locus of the constant term zero $G_{\0}$-invariant polynomials on $V$ where $V$ is 
a rational $G_{\0}$-module. Kac's results are in a more general context than our paper and there is some overlap in our results. We anticipate that the varieties ${\mathcal N}$ will play an  
important role in the representation theory for Lie superalgebras. There is evident in Section~\ref{S:connectionswithX} with the strong connections with the Duflo-Serganova commuting varieties
that were introduced in the mid 2000's. In the late 2000's, Gruson and Leidwanger \cite{GL} investigated the case of the nilpotent cone in the orthosymplectic case and constructed a resolution of 
singularities. The aim of our paper is to present a unified self-contained treatment of ${\mathcal N}$ for classical simple Lie superalgebras that can be easily referenced by those working in super representation theory. 

\subsection{} The paper is organized as follows. In Section~\ref{S:prelims}, the nilpotent cone for classical Lie superalgebras is defined. We also indicate 
how this definition generalizes the definition of the nilpotent cone for complex semisimple Lie algebras. Our first main result (Theorem~\ref{t:glmn case}) in Section~\ref{S:orbitsgl} demonstrates that the nilpotent cone for $\mathfrak{gl}(m|n)$ has finitely many $G_{\0}$-orbits. Explicit orbit representatives are determined. 
In Section~\ref{S:orbitsgeneral}, we prove a theorem that allows us to extend the finiteness result to the nilpotent cone for other classical simple Lie superalgebras. 
The ideas of the theorem are originally due to Richardson and can be applied in cases when there is a suitable embedding of the classical simple 
Lie superalgebras into a general linear Lie superalgebra. With these tools, we show that for classical simple Lie superalgebras, 
${\mathcal N}$ has finitely many $G_{\0}$-orbits. 

Duflo and Serganova \cite{DS} introduced the commuting variety ${\mathcal X}$ for any finite-dimensional Lie superalgebra. They proved that 
for basic classical Lie superalgebras, ${\mathcal X}$ has finitely many $G_{\0}$-orbits. In Section~\ref{S:connectionswithX}, we prove that for all classical simple 
Lie superalgebras, one has an inclusion ${\mathcal X}\subseteq {\mathcal N}$. In this way, one should consider the nilpotent cone a larger algebraic 
variety whose geometric properties should encompass that of ${\mathcal X}$. We show that our results on the finiteness of orbits for ${\mathcal N}$ allows us to extend the 
finiteness results in \cite{DS} to a wider class of Lie superalgebras (cf. Corollary~\ref{c:finitenessX}). 

\subsection{Acknowledgements} We thank Vera Serganova for comments and suggestions on an earlier version of our manuscript.

\section{Preliminaries}\label{S:prelims}

\subsection{Notation: }  Throughout this paper we will use the conventions in \cite{BKN1, BKN2, BKN3, GGNW}. 
Let $\mathfrak{g} = \mathfrak{g}_{\0} \oplus \mathfrak{g}_{\1}$ be a {\em classical Lie superalgebra} 
over $k={\mathbb C}$. This means there exists $G_{\0}$, a corresponding connected reductive algebraic group, such
 that $\text{Lie }G_{\0} = \mathfrak{g}_{\0}$ where $\mathfrak{g}_{\1}$ is a $G_{\0}$-module via the adjoint action.
The Lie superalgebra ${\mathfrak g}$ is a \emph{basic classical} if it is a classical Lie superalgebra with a nondegenerate invariant supersymmetric even bilinear form. 

Let ${\mathfrak g}$ be a classical Lie superalgebra and $S^{\bullet}({\mathfrak g}_{\1}^{*})$ be the symmetric algebra on the dual of 
${\mathfrak g}_{\1}$. We will often regard $S^{\bullet}({\mathfrak g}_{\1}^{*})$ as the polynomial functions on ${\mathfrak g}_{\1}$. 
Let  $S^{\bullet}({\mathfrak g}_{\1}^{*})_{+}$ be the polynomials with constant term equal to zero. The algebraic group $G_{\0}$ acts on ${\mathfrak g}_{\1}$, 
so we can consider the $G_{\0}$-invariant polynomials with zero constant term on ${\mathfrak g}_{\1}$ denoted by $S^{\bullet}({\mathfrak g}_{\1}^{*})_{+}^{G_{\0}}$. 
The {\em nilpotent cone}, ${\mathcal N}$, for ${\mathfrak g}$ is the zero set of these polynomials: 
$${\mathcal N}=Z(S^{\bullet}({\mathfrak g}_{\1}^{*})_{+}^{G_{\0}}).$$ 
Observe that $\mathcal{N} \subseteq \mathfrak{g}_{\1}$. The algebraic variety $\mathcal{N}$ is a $G_{\0}$-invariant closed cone 
in ${\mathfrak g}_{\1}$. 

\subsection{Simple Classical Lie Superalgebra} The main results of the paper will be stated for classical ``simple'' Lie superalgebras. 
We will use the term simple Lie superalgebra to refer to the Lie superalgebra of general interest that are not simple in the true sense, 
but close enough to being simple (cf. \cite{GGNW}). The Lie superalgebras that will be considered ``simple" include: 
\begin{itemize} 
\item $\mathfrak{gl}(m|n)$, $\mathfrak{sl}(m|n)$, $\mathfrak{psl}(n|n)$
\item $\mathfrak{osp}(m,n)$
\item $D(2,1,\alpha)$ 
\item $G(3)$
\item $F(4)$ 
\item $\mathfrak{q}(n)$, $\mathfrak{psq}(n)$ 
\item ${\mathfrak p}(n)$, $\widetilde{{\mathfrak p}}(n)$.
\end{itemize} 
For the Lie superalgebras of Type Q, $\mathfrak{q}(n)$ is the Lie superalgebra with even and odd parts $\mathfrak{gl}(n)$, while 
$\mathfrak{psq}(n)$ is the corresponding simple subquotient of $\mathfrak{q}(n)$ (cf. \cite{PS}). The Lie superalgebras 
that fall into the family of Type P include ${\mathfrak p}(n)$ and its enlargement $\widetilde{{\mathfrak p}}(n)$.

\subsection{Generalization of the ordinary nilpotent cone} 

We now indicate how our results generalize known results for the nilpotent cone for complex semisimple Lie algebras. 
Let ${\mathfrak a}$ be a complex semisimple Lie algebras, and set ${\mathfrak g}={\mathfrak g}_{\0}\oplus {\mathfrak g}_{\1}$ with 
${\mathfrak g}_{\0}={\mathfrak a}={\mathfrak g}_{\1}$ as vector spaces. We can make ${\mathfrak g}$ into a Lie superalgebra 
by defining the bracket on ${\mathfrak g}_{\0}$ to be the ordinary Lie bracket on ${\mathfrak a}$. The bracket of an element in 
${\mathfrak g}_{\0}$ on ${\mathfrak g}_{\1}$ is given by the adjoint action, and the bracket of any two elements in ${\mathfrak g}_{\1}$ is 
zero. 

Let $G_{\0}$ be the semisimple simply connected group such that $\text{Lie }G_{\0}={\mathfrak g}_{\0}$. Then 
${\mathcal N}$ is the ordinary nilpotent cone for ${\mathfrak g}_{\0}={\mathfrak a}$. One can also set this up for fields of 
prime characteristic, if one considers Lie algebras that arise as the Lie algebra of a semisimple algebraic group.


\section{$G_{\0}$-orbits on $\mathcal{N}$: $\mathfrak{gl}(m|n)$ case}\label{S:orbitsgl}

\subsection{}\label{SS:actions}  The adjoint action of $G_{\0} = GL_m(\C) \times GL_n(\C) $ on $\mathfrak{g}_{\1}$ is given by conjugation.  Explicitly, 
 \[ 
 \left[
\begin{array}{c|c}
A  & 0 \\ \hline
 0 & B
\end{array}\right] \cdot \left[
\begin{array}{c|c}
0  & X^+ \\ \hline
 X^- & 0
\end{array}\right] = \left[
\begin{array}{c|c}
A  & 0 \\ \hline
 0 & B
\end{array}\right]^{-1}  \left[
\begin{array}{c|c}
0  & X^+ \\ \hline
 X^- & 0
\end{array}\right] \left[
\begin{array}{c|c}
A  & 0 \\ \hline
 0 & B
\end{array}\right] =\left[
\begin{array}{c|c}
0 & A^{-1}X^+B \\ \hline
 B^{-1}X^-A & 0
\end{array}\right].
 \]
 
In this case, \cite[Section 2.6]{Fuks} has determined the generators of the invariants, $S({\mathfrak g}_{\1}^{*})^{G_{\0}}_{+}$, to be 
\[Tr((X^+X^-)^k), \hspace{0.5cm} k=1,\ldots, l\]
where $l=\text{min}\{m,n\}$.

For the following theorem we recall the notion of a matrix in column echelon form.  A matrix is in (reduced) column echelon form if it satisfies the following conditions:

\begin{itemize}
    \item all columns that consist entirely of zero entries appear as the right most columns of the matrix
    \item the first nonzero entry of each column is called the pivot, and the pivot is the only nonzero entry in its row
    \item if $j>i$, then the pivot of the nonzero column $c_j$ lies in a row strictly below the row of the pivot of column $c_i$.
\end{itemize}
We also use the convention that each pivot element is 1. By transposing the matrix, the notion of row echelon form can defined in a similar way 

In this section we will show that $\mathfrak{gl}(m\lvert n)$ has finitely many $G_{\0}$-orbits. Furthermore, explicit orbit representatives 
for this action will be exhibited. The results are summarized in the following theorem and the proof will be given in the next section. 

\begin{theorem} \label{t:glmn case}
Let $\mathfrak{g} = \mathfrak{gl}(m \lvert n).$  
\begin{itemize} 
\item[(a)] The number of $G_{\0}$-orbits of the adjoint action on $\mathcal{N}$ is finite.
\item[(b)] The complete set of orbit representatives is given by matrices 
$$
Y=\left[
\begin{array}{c|c}
0  & Y^+ \\ \hline
 Y^- & 0
\end{array}
\right] 
$$
where  
\[Y^+ = \left[
\begin{array}{c|c}
I_r  & 0 \\ \hline
 0 & 0
\end{array}\right] 
\qquad \text{and}  \qquad 
Y^{-}=
\left[
\begin{array}{c c c|c c c}
  J_1 & & & | & 0 & 0 \\ 
   & \ddots &  & C_{r_1} & 0 & 0\\ 
  & & J_t & | & 0 & 0 \\ \hline
  - & R_{r_2} & - & 0 & 0 & 0 \\
  0 & 0 & 0 & 0 & I_s & 0 \\
  0 & 0 & 0 & 0 & 0 & 0  \\
\end{array}\right].
\]
Here $I_{r}$ (resp. $I_{s}$) is a $r\times r$ (resp. $s\times s$) identity matrix, $J_{1},\dots J_{t}$ are Jordan blocks with zero eigenvalues, where the Jordan block $J_{i}$ is of size $k_{i}\times k_{i}$, $i=1,2,\dots,t$ with $k_{1}\geq k_{2}\dots \geq k_{t}$ with $r=\sum_{i=1}^{t}k_{i}$. Furthermore, the matrix $C_{r_1}$ (resp. $R_{r_2}$) are of size $r \times r_1$ (resp. $r_2\times r$) and of the form 
\[ C_{r_1} = \begin{pmatrix} 
e_{i_1} & e_{i_2} & \cdots & e_{i_{r_1}}
\end{pmatrix},\ \ \ \  R_{r_2} = \begin{pmatrix}
e_{j_1} & e_{j_2} & \cdots & e_{j_{r_2}}
\end{pmatrix}^T\]
where $e_{k}$ is the column vector with a single $1$ in the $k$-th row and zeroes elsewhere. Each index $i_p$ (resp. $j_{q}$) belongs to the set 
$\{k_1, k_1+k_2, \ldots , k_1+k_2+\ldots + k_t \}$ (resp. $\{ 1, 1+k_1, \ldots, 1+k_1+\ldots +k_{t-1} \} $). 
\end{itemize} 
\end{theorem}

\subsection{Proof of Theorem~\ref{t:glmn case}} The proof of theorem in the prior section will entail several steps. We start will a general 
element $X\in {\mathcal N}$ where ${\mathfrak g}={\mathfrak g}{\mathfrak l}(m|n)$. Through a series of conjugations (i.e., applications of elements in 
$G_{\0}$), the element $X$ will be transformed into a matrix $Y$ of the form in Theorem~\ref{t:glmn case}(b). In the process of this 
transformation, we will use $Y$ (including $Y^{+}$ and $Y^{-}$) to denote the current matrix under the series of transformations. 

\subsubsection{} Let $X\in {\mathcal N}$. Using standard results in linear algebra (involving equivalence of matrices), 
and the action of $(A,B)\in G_{\0}$ on $X$, there exists $(A,B)\in G_{\0}$ such that $A^{-1}X^+B=Y^{+}$ where 
\[Y^+ = \left[
\begin{array}{c|c}
I_r  & 0 \\ \hline
 0 & 0
\end{array}\right] \]
where  $r$ is the rank of $X^+$. \\

\subsubsection{} The next step is to identify and later work with $(A,B)\in G_{\0}$ that centralize $Y^{+}$, which is equivalent to 
 the condition: $A^{-1}Y^+B=Y^{+}$. In order elaborate further,  it will be useful to consider $(A,B)$ in block matrix form
$$
A=\left[
\begin{array}{c|c}
A_{11}  & A_{12} \\ \hline
A_{21} & A_{22}
\end{array} \right] \quad 
B= \left[
\begin{array}{c|c}
B_{11}  & B_{21}\\ \hline
 B_{21} & B_{22}
\end{array}\right]. 
$$
The centralizing condition is equivalent to 
$$
\left[\begin{array}{c|c}
B_{11}  & B_{12} \\ \hline
 0 & 0
\end{array}\right] = \left[
\begin{array}{c|c}
A_{11}  & 0 \\ \hline
 A_{21} & 0
\end{array}\right].
$$
Therefore, one has 
$$
A=\left[
\begin{array}{c|c}
A_{11}  & A_{12} \\ \hline
 0 & A_{22}
\end{array}\right]\ \ B= \left[
\begin{array}{c|c}
A_{11}  & 0 \\ \hline
 B_{21} & B_{22}
\end{array}\right]. 
$$
Note we have a formula for $B^{-1}$ in terms of blocks:
$$B^{-1} = \left[
\begin{array}{c|c}
A_{11}^{-1}  & 0 \\ \hline
 -B_{22}^{-1}B_{21}A_{11}^{-1} & B_{22}^{-1}
\end{array}\right]. 
$$
We can now provide a formula for the action of $(A,B)$ in the centralizer of $Y^{+}$ on $Y^-$ in block form: 

\begin{equation*}
\begin{split}    
B^{-1}Y^-A &= \left[
\begin{array}{c|c}
A_{11}^{-1}  & 0 \\ \hline
 -B_{22}^{-1}B_{21}A_{11}^{-1} & B_{22}^{-1}
\end{array}\right]  \left[
\begin{array}{c|c}
Y_{11}^-  & Y_{12}^- \\ \hline
Y_{21}^- & Y_{22}^-
\end{array}\right] \left[
\begin{array}{c|c}
A_{11}  & A_{12} \\ \hline
 0 & A_{22}
\end{array}\right] \\
&= \left[
\begin{array}{c|c}
A_{11}^{-1}Y_{11}^-A_{11}  & A_{11}^{-1}(Y_{11}^-A_{12}+Y_{12}^-A_{22}) \\ \hline
 B_{22}^{-1}(Y_{21}^-A_{11}-B_{21}A_{11}^{-1}Y_{11}^-A_{11} ) & B_{22}^{-1}(Y_{21}^-A_{12}+Y_{22}^-A_{22})-B_{22}^{-1}B_{21}A_{11}^{-1}(Y_{11}^-A_{12}+Y_{12}^-A_{22})
\end{array}\right].
\end{split}
\end{equation*} 
As long as we work with $(A,B)\in G_{\0}$ that centralize $Y^{+}$ (i.e., $(A,B)\in c_{G_{\0}}(Y^{+})$), we can focus on transforming $Y^{-}$ into the desired form. 

\subsubsection{} Observe that $Y_{11}^-$ can be put into its Jordan form $J$ (upper triangular) via $A_{11}$. If one chooses $A_{11}$ to be in the centralizer of $J$, one can replace both $Y_{11}^-$ and $A_{11}^{-1}Y_{11}^-A_{11}$ with $J$  in the expression above. By operating with elements $(A,B)\in c_{G_{\0}}(Y^{+})$ with $A_{11}$ that centralizes $J$, our new expression for the action on $Y^{-}$ is

\[ B^{-1}Y^-A = \left[
\begin{array}{c|c}
J  & A_{11}^{-1}(JA_{12}+Y_{12}^-A_{22}) \\ \hline
 B_{22}^{-1}(Y_{21}^-A_{11}-B_{21}J ) & B_{22}^{-1}(Y_{21}^-A_{12}+Y_{22}^-A_{22})-B_{22}^{-1}B_{21}A_{11}^{-1}(JA_{12}+Y_{12}^-A_{22})
\end{array}\right]. \]

\subsubsection{} We also observe that all the eigenvalues of $J$ are zero since 

\[Y^+Y^- = \left[
\begin{array}{c|c}
I_r  & 0 \\ \hline
 0 & 0
\end{array}\right] \left[
\begin{array}{c|c}
J  & X_{12}^- \\ \hline
 Y_{21}^- & Y_{22}^-
\end{array}\right] = \left[
\begin{array}{c|c}
J  & Y_{12}^- \\ \hline
 0 & 0
\end{array}\right]. \]

The condition that $Tr((Y^+Y^-)^k)=0$ for $ k=1,\ldots, m$ where $m=\text{min}\{m,n\}$ implies that 
if $\lambda_{1},\lambda_{2},\dots,\lambda_{r}$ are the eigenvalues of $J$ then 
$\lambda_{1}^{k}+\lambda_{2}^{k}+\dots + \lambda_{r}^{k}=0$ for $k=1,2,\dots m$. Since $r\leq m$ this implies that 
$\lambda_{j}=0$ for all $j$. 

\subsubsection{} Let $J$ be a Jordan canonical form with Jordan blocks $J_{1}, J_{2},\dots, J_{t}$ of sizes $k_{1}\geq k_{2} \geq \dots \geq k_{t}$ with 
all Jordan blocks having eigenvalue zero: 
\[J = \begin{pmatrix} 
J_{1} & 0 & 0 & \cdots & 0 \\
0 & J_{2} & 0 & \cdots & 0 \\
\vdots & & \ddots &  & \\
0 & 0 & 0 & \cdots & 0\\
0 & 0 & 0 & \cdots & J_{t}
\end{pmatrix}.\]

Consider the action of $(A,B)\in c_{G_{\0}}(Y^{+})$ with $A_{11}$, $A_{22}$, and $B_{22}$ to be (appropriately sized) identity matrices. Then 

\[ B^{-1}Y^-A = \left[
\begin{array}{c|c}
J  & JA_{12}+Y_{12}^{-}\\ \hline
Y_{21}^--B_{21}J  & *
\end{array}\right]. \]

The matrix $JA_{12}$ (resp. $B_{21}J)$ consists of matrix entries with zeros in rows (resp. columns) $k_{1}$, $k_{1}+k_{2}$,\dots, $k_{1}+k_{2}+\dots +k_{t}$ 
(resp. $1$, $1+k_{1}$, $1+k_{1}+k_{2}$, \dots, $1+k_{1}+k_{2}+\dots+k_{t-1}$) and 
arbitrary entries in the other rows (resp. columns). Therefore, one can choose entries in $A_{12}$ (resp. $B_{21}$) to make 
$JA_{12}+Y_{12}^-$ (resp. $Y_{21}^--B_{21}J$) into a matrix with possibly non-zero entries in rows (resp. columns) $k_{1}$, $k_{1}+k_{2}$,\dots, $k_{1}+k_{2}+\dots +k_{t}$ 
(resp. $1$, $1+k_{1}$, $1+k_{1}+k_{2}$, \dots, $1+k_{1}+k_{2}+\dots+k_{t-1}$) and zeros in the other rows (resp. columns).

\subsubsection{}\label{ssec:3.2.6} Now consider the action of $(A,B)\in c_{G_{\0}}(Y^{+})$ with $A_{11}$ being the identity matrix, $A_{12}=0$ and $B_{21}=0$. Then 
\[ B^{-1}Y^-A = \left[
\begin{array}{c|c}
J  & Y_{12}^{-}A_{22}\\ \hline
B_{22}^{-1}Y_{21}^{-}  & *
\end{array}\right]. \]
Now one can make $A_{22}$ (resp. $B_{22}^{-1}$) into a product of a permutation matrix and elementary matrices to transform 
 $Y_{12}^{-}A_{22}$ (resp. $B_{22}^{-1}Y_{21}^{-}$) into a column (resp. row) echelon form with possibly non-zero entries in rows (resp. columns) $k_{1}$, $k_{1}+k_{2}$,\dots, $k_{1}+k_{2}+\dots +k_{t}$ (resp. $1$, $1+k_{1}$, $1+k_{1}+k_{2}$, \dots, $1+k_{1}+k_{2}+\dots+k_{t-1}$) and zeros in the other rows (resp. columns).

\subsubsection{} The next step is to transform $Y_{12}^{-}$ into a matrix $C_{r_1}$ of the form stated in the theorem.  
Let $(A,B)\in c_{G_{\0}}(Y^{+})$ with 
$A_{12}=0$, $B_{21}=0$ and $A_{22}$, $B_{22}$ being identity matrices. Then 

\[ B^{-1}Y^-A = \left[
\begin{array}{c|c}
A_{11}^{-1}JA_{11}  & A_{11}^{-1}Y_{12}^{-}\\ \hline
Y_{21}^-A_{11} & *
\end{array}\right]. \]

By using the action of $A_{11}\in c(J)$ (centralizer of $J$) without loss of generality we can assume that the pivots in $Y_{12}^{-}$ are all $1$'s. Normally we can transform $Y_{12}^{-}$ into a 
elementary row echelon form by making $A_{11}^{-1}$ into a product of elementary matrices that performs the row operations used in Gaussian elimination, and since $Y_{12}^-$ is already in column echelon form from the previous step, this is equivalent to the form of $C_{r_1}$. 
The issue is that this product of elementary matrices need not centralize $J$. This problem can be remedied for the cases of the elementary row operations of row addition and row scaling as follows.\footnote{The permutation matrices necessary to perform the row swapping operation of Gaussian elimination are also not in the centralizer of $J$, however there are only finitely many possible forms of $C_{r_1}$ up to permutation.} 

The matrix $Y_{12}^{-}$ is in column echelon form and has non-zero entries in rows $k_{1}$, $k_{1}+k_{2}$,\dots, $k_{1}+k_{2}+\dots +k_{t}$ with zeros in the other rows. 
Set $f(j)=k_{1}+k_{2}+\dots +k_{j}$,  $j=1,2,\dots, t$. Furthermore, 
the non-zero pivots can lie in matrix positions $(f(j),j)$ for $j=1,2,\dots, t$. We want to use $A_{11}^{-1}$ to clear the matrix entries in the columns 
directly below the pivots. If there was no restriction on $A_{11}^{-1}$ one can accomplish this with a product of elementary matrices. For example, 
if one wants to eliminate a non-zero entry $\alpha$ in position $(f(i),j)$, then one can apply the elementary matrix $E_{f(j),f(i)}(-\alpha)$ on the left which 
replaces Row $f(j)$ with $-\alpha \times$ Row $f(j)$+ Row $f(i)$. 

Consider the matrix $L_{f(j),f(i)}(-\alpha)$ which is in the centralizer of $J$ where 
$$L_{f(j),f(i)}(-\alpha)= \begin{pmatrix} 
I_{k_{1}} & 0 & 0 & \cdots & 0 \\
Z_{2,1} & I_{k_{2}} & 0 & \cdots & 0 \\
Z_{3,1} & Z_{3.2} & I_{k_{3}} &\dots  &0 \\
\vdots & \vdots  & \vdots & \ddots & 0\\
Z_{t,1} & Z_{t,2} & Z_{t,3} & \cdots & I_{k_{t}}.
\end{pmatrix}
$$
Here $I_{k_{j}}$ are $k_{j}\times k_{j}$ identity matrices and $Z_{j,i}$ are $k_{j} \times k_{i}$ matrices where $i=1,2,\dots,t-1$ and $j=2,3,\dots, t$. Moreover, 
set $Z_{j^{\prime},i^{\prime}}=0$ for $(j,i)\neq (j^{\prime},i^{\prime})$, and the block matrix 
$$Z_{j,i}= \begin{pmatrix} 
0 & -\alpha I_{k_{j}}
\end{pmatrix}.
$$
It can be directly verified that $L_{f(j),f(i)}(-\alpha)$ will have the same effect as $E_{f(j),f(i)}(-\alpha)$ on the Gaussian elimination process on $Y_{12}^{-}$. 

Similarly, it can be directly verified that the matrix
\[ M_{k_j} (\alpha) = \begin{pmatrix} 
I_{k_{1}} & 0 & 0 & \cdots & 0 \\
0 & \ddots & 0 & \cdots & 0 \\
0 & 0 & \alpha I_{k_j} &\dots  &0 \\
\vdots & \vdots  & \vdots & \ddots & 0\\
0 & 0 & 0 & \cdots & I_{k_{t}}.
\end{pmatrix} \]
will scale row $k_1+\ldots+k_j$ by $\alpha$ and also centralizes $J$.

Therefore, if the rank of the current iteration of $Y_{12}^{-}$ is $r_1$ with pivots in rows $i_1,\ldots, i_{r_1}$, then it is transformed to
\[ \begin{pmatrix} 
C_{r_1} & 0 & \cdots & 0
\end{pmatrix} \]
where
\[ C_{r_1} = \begin{pmatrix} 
e_{i_1} & e_{i_2} & \cdots & e_{i_{r_1}} 
\end{pmatrix} \]
and where $e_{i_p}$ is the column vector with a $1$ in the $i_p$-th row and zeroes elsewhere.  Furthermore, each index $i_p$ is an element of $\{k_1, k_1 + k_2, \ldots, k_1+ k_2 + \ldots + k_t \}$.  That is, the nonzero entry of each column $e_{i_p}$ occurs in one of the aforementioned rows.

\subsubsection{} We now perform that same procedure as in the last step to  to transform $Y_{21}^{-}$ into a matrix $R_{r_2}$ with form analogous to the transpose of $C_{r_1}$.  Recall from Section \ref{ssec:3.2.6} that we have already transformed $Y_{21}^-$ into row echelon form.
Let $(A,B)\in c_{G_{\0}}(Y^{+})$ with $A_{12}=0$, $B_{21}=0$ and $A_{22}$, $B_{22}$ being identity matrices. 
This time one makes $A_{11}$ into a product of matrices that are upper triangular, are in the centralizer of $J$, and perform 
column operations to clear out the entries in the pivot rows of all non-zero entries (except for the pivot). In this case each index $j_q$ is an element of $\{1+k_1, 1+k_1 + k_2, \ldots, 1 + k_1 + k_2 + \ldots + k_{t-1} \}$ and so the nonzero entry of each row occurs in one of these columns.

Note that in the process we have changed $Y_{12}^-$ out of the form of $C_{r_1}$. Consider $A_{11}$. Then $A_{11}^{-1}$ will still be upper triangular and $A_{11}^{-1}Y_{12}^{-}$ will be a matrix with $1$'s in the same pivot entries as $Y_{12}^-$ with possible non-zero entries above and to the right of the position of the pivot entries in $Y_{12}^-$. 
The next two steps will correct this issue. 

\subsubsection{} Now let $(A,B)\in c_{G_{\0}}(Y^{+})$ with $A_{11}$, $A_{22}$, and $B_{22}$ be identity matrices, and $B_{21}=0$.
Then 
\[ B^{-1}Y^-A = \left[
\begin{array}{c|c}
J  & JA_{12}+Y_{12}^{-}\\ \hline
Y_{21}^{-} & *
\end{array}\right]. \]
We can now transform $Y_{12}^{-}$ into a matrix in row echelon form by choosing $A_{12}$ to kill the non-zero entries in the ``non-pivot'' rows. Note that 
$Y_{21}^{-}$ is unchanged. 

\subsubsection{} Let $(A,B)\in c_{G_{\0}}(Y^{+})$ with $A_{11}$, $B_{22}$ being identity matrices, and $A_{12}=0$, $B_{21}=0$. 
Then 
\[ B^{-1}Y^-A = \left[
\begin{array}{c|c}
J  & Y_{12}^{-}A_{22}\\ \hline
Y_{21}^{-} & *
\end{array}\right]. \]
The matrix $A_{22}$ can be chosen to clear out the non-zero entries that are not pivots to make $Y_{12}^{-}$ in column echelon form. Again $Y_{21}^{-}$ is 
unchanged in the process.

\subsubsection{} The final step is to transform $Y_{22}^{-}$ into a matrix of the desired form. We will work with $(A,B)$ in the centralizer of $Y^{+}$ with 
$A_{11}$ centralizing $J$. Let $A_{11} = I$ and $A_{12}, B_{21} = 0$.   Then $Y_{22}^-$ is transformed to $B_{22}^{-1}Y_{22}^-A_{22}$.  
Let $\text{rank}(C_{r_1}) = r_1$ and 
$\text{rank}(R_{r_2}) = r_2$. 

Write
\[Y_{22}^- =   \begin{pmatrix} 
\xi_{11} & \xi_{12}  \\
\xi_{21} & \xi_{22}
\end{pmatrix}\]
where $\xi_{11}$ is a $r_2\times r_1$ matrix. Moreover, $\xi_{12}$ is $r_2 \times (n-r-r_1)$, $\xi_{21}$ is $(m-r-r_2) \times r_1$, and $\xi_{22}$ is $(m-r-r_2) \times (n-r-r_1)$.

Now to centralize $C_{r_1}$ and $R_{r_2}$ we require the first $r_1$ rows of $A_{22}$ to be the first $r_1$ rows of the identity matrix and similarly the first $r_2$ columns of $B_{22}^{-1}$ to be the first $r_2$ columns of the identity matrix.  Therefore we can write $A_{22}$ and $B_{22}^{-1}$ in block form (in a similar way with $Y_{22}^-$) as
\[A_{22} =   \begin{pmatrix} 
I_{r_1} & 0  \\
\alpha_{21} & \alpha_{22}
\end{pmatrix} \text{ and } B_{22}^{-1} =   \begin{pmatrix} 
I_{r_2} & \beta_{12}  \\
0 & \beta_{22}
\end{pmatrix}. \]
Then
\[B_{22}^{-1}Y_{22}^-A_{22} = \begin{pmatrix} 
\xi_{11}+\beta_{12}\xi_{21}+\xi_{12}\alpha_{21}+\beta_{12}\xi_{22}\alpha_{21} & \xi_{12}\alpha_{22}+\beta_{12}\xi_{22}\alpha_{22}  \\
\beta_{22}\xi_{21}+\beta_{22}\xi_{22}\alpha_{21} & \beta_{22}\xi_{22}\alpha_{22}
\end{pmatrix}.  \]
This shows that we can send $\xi_{22}$ to
\[ \begin{pmatrix} 
I_{s} & 0  \\
0 & 0
\end{pmatrix} \]
where $s = \text{rank}(\xi_{22}).$  

\subsubsection{} Now choose $(A,B)\in c_{G_{\0}}(Y^{+})$ with $A_{11}, A_{22}$ and $B_{22}^{-1}$ to be the identity, while allowing $A_{12}$ and $B_{21}$ to be free. 
This action will fix $\xi_{22}$ and image of $Y_{22}^{-}$ is 
\begin{equation}\label{eq:nexttolast}
-B_{21}JA_{12} - B_{21}Y_{12}^{-}+Y_{21}^{-}A_{12} + Y_{22}^{-}. 
\end{equation}

Set $B_{21}$ to be zero.  Let the pivots of $Y_{21}^{-}$ be in columns $j_1, \ldots, j_{r_2}$.  Then set the $j_1$-th row of $A_{12}$ to be the negative of the 1st row of $Y_{22}^{-}$, the $j_2$-th row of $A_{12}$ to be the negative of the 2nd row of $Y_{22}^{-}$ and so forth.  Then by (\ref{eq:nexttolast}), $(A,B)$ 
sends $\xi_{11}$ and $\xi_{12}$ to 0.

\subsubsection{} Finally, we need to choose $(A,B)\in c_{G_{\0}}(Y^{+})$ that stabilizes the current $Y^{-}_{11}, Y^{-}_{12}, Y^{-}_{21}$, $\xi_{11}, 
\xi_{12}, \xi_{22}$ and sends  $\xi_{21}$ to 0.  This can be accomplished by  setting  $A_{11}, A_{22}$ and $B_{22}^{-1}$ to be the identity, and $A_{12}$ to be zero.  
If the pivots of $Y_{12}^{-}$ are in rows $i_1, \ldots, i_{r_1}$.  Then set the $i_1$-th column of $B_{21}$ to be the 1st column of 
$Y_{22}^{-}$, the $i_2$-th column of $B_{21}$ to be the 2nd column of $Y_{22}^{-}$ and so forth. This concludes the proof of Theorem~\ref{t:glmn case}.


\section{$G_{\0}$-orbits on $\mathcal{N}$: general case}\label{S:orbitsgeneral}

\subsection{} In the ordinary Lie algebra case, the adjoint action of the algebraic group $G$ on $\mathfrak{g}$ is known to have finitely many nilpotent orbits via Richardson's Theorem (see \cite[Theorem 3.8]{Hum}).  In this section, we prove an appropriate generalization for Lie superalgebras.  We begin by stating the following lemma whose proof can be found in \cite[Section 2.4]{Jan2}.

\begin{lemma} \label{L:preRichardson}
Let $G$ be an algebraic group and let $H$ be a closed subgroup of $G$.  Let $X$ be a $G$-variety and let $Y$ be a closed and $H$-invariant subvariety of $X$.  Suppose that for all $y \in Y$, 
\begin{equation} \label{eq:tanspace}
T_y(G\cdot y) \cap T_y(Y) \subseteq (d\pi_y)_{id}(\operatorname{Lie }H), 
\end{equation}
where $\pi_y: G \to G\cdot y$ sends $g$ to $gy$.
Then the intersection with $Y$ of each $G$-orbit in $X$ is a finite union of $H$-orbits.
\end{lemma}

\subsection{Generalization of Richardson's Theorem} We can now state a generalization of Richardson's theorem in the context of Lie superalgebras. 

\begin{theorem} \label{T:superrich}
 Let $G_{\0}$ be a closed subgroup of some $GL_m(\C) \times GL_n(\C)$.  Let $\mathfrak{g} = \mathfrak{g}_{\0} \oplus \mathfrak{g}_{\1}$ be a Lie superalgebra with 
 $\operatorname{Lie }G_{\0} = \mathfrak{g}_{\0}$.  Suppose there exists a supersubspace $M \subseteq \mathfrak{gl}(m | n)$ such that $\mathfrak{gl}(m | n) = M \oplus \mathfrak{g} $ and $[\mathfrak{g}, M] \subseteq M$.  Then the intersection with $\mathfrak{g}$ of each $GL_m(\C) \times GL_n(\C)$-orbit in $\mathfrak{gl}(m \lvert n)$ is a union of finitely many $G_0$-orbits. 
\end{theorem}

\begin{proof}
Assume such a complement $M$ exists.  We show that condition (\ref{eq:tanspace}) of Lemma~\ref{L:preRichardson} is satisfied where $X = \mathfrak{gl}(m | n)$, $Y=\mathfrak{g}$, $G=GL_m(\C) \times GL_n(\C)$, and $H = G_{\0}$.  This means one must show for every $y \in \mathfrak{g}$
\[ T_y(G \cdot y) \cap T_y (\mathfrak{g}) \subseteq (d\pi_y)_{id}(\mathfrak{g}_{\0}).  \]
We observe that $T_y(\mathfrak{g}) = \mathfrak{g}$ and $(d\pi_y)_{id}(\mathfrak{g}_{\0}) = [\mathfrak{g}_{\0}, y]$.  

Now showing (\ref{eq:tanspace}) is equivalent in our case to showing 
\begin{equation} \label{eq:tanspace2}
T_y (G \cdot y) \cap \mathfrak{g} \subseteq [\mathfrak{g}_{\0}, y], 
\end{equation} 
for every $y \in \mathfrak{g}.$  

Note that $\text{Lie }G_x =(\mathfrak{gl}(m)\times \mathfrak{gl}(n))_x$,  and by a standard fact in \cite[Prop. 9.1]{Bor} this is equivalent to $T_y(G \cdot y) = [\mathfrak{gl}(m)\times \mathfrak{gl}(n), y]$.\\

Now apply the complement condition to obtain that
\[ [\mathfrak{gl}(m)\times \mathfrak{gl}(n), y] = [M_{\0} \oplus \mathfrak{g}_{\0},y] = [M_{\0},y] + [\mathfrak{g}_{\0},y].\]
By assumption $[M_{\0},y] \subseteq M$, so (\ref{eq:tanspace2}) becomes
\[T_y(G \cdot y) \cap \mathfrak{g} \subseteq (M + [\mathfrak{g}_{\0},y]) \cap \mathfrak{g} \]
and since $M \cap \mathfrak{g} = \{0\}$ and $[{\mathfrak g}_{\0},y]\subseteq {\mathfrak g}$, this reduces to
$$ 
T_y(G \cdot y) \cap \mathfrak{g} \subseteq [\mathfrak{g}_{\0},y]. 
$$
So (\ref{eq:tanspace2}) is satisfied and the result follows by Lemma~\ref{L:preRichardson}.
\end{proof}

\subsection{} \label{SS:embeddings}  In the case when Richardson's Theorem is applied to show the finiteness of $G$-orbits for the nilpotent cone for complex semisimple Lie algebras, 
the existence of such an $M$ is guaranteed via complete reducibility. More specifically, this follows from regarding $\mathfrak{gl}(n)$ as a a $G$-module under the adjoint 
action.  Then $\mathfrak{g}=\text{Lie }G$ is a submodule and therefore has a vector space complement $M$ in $\mathfrak{gl}(n)$ that is invariant under the adjoint action of $\mathfrak{g}$. 

In the situation for Lie superalgebras, we can apply this same reasoning to $G_{\0}$ acting on $\mathfrak{g}_{\0}$ to produce an $M_{\0}$ satisfying $\mathfrak{gl}(m \lvert n)_{\0} = \mathfrak{g}_{\0} \oplus M_{\0}$ and $G_{\0} \cdot M_{\0} \subseteq M_{\0}$.  Then $M_{\0}$ will also be invariant under the derived $\mathfrak{g}_{\0}$ action, so that $[\mathfrak{g}_{\0},M_{\0}] \subseteq M_{\0}$. However, an issue arises when considering $\mathfrak{g}_{\1}$.  We can still regard $\mathfrak{g}_{\1}$ as a $G_{\0}$-module and produce a complement $M_{\1}$, but since the derived action involves only $\mathfrak{g}_{\0}$, we know nothing about $[\mathfrak{g}_{\1}, M_{\1}]$.   

In order to prove finiteness of orbits for ${\mathcal N}$ in the superalgebra case we construct $M_{\1}$ in a case-by-case manner 
and show directly that $[\mathfrak{g}_{\1}, M_{\1}] \subseteq M_{\0}$.  Details can be found in Section~\ref{S:complements} (Appendix).  It is worth noting that the methods we use to produce these complements are analogous to the methods used to produce complements in the characteristic $p$ case of Richardson's Theorem.  Therefore, even though all of the Lie superalgebras considered here are over $\C$, we still need to use ideas from the characteristic $p$ case in order to produce compatible complements.

\subsection{} We can now verify the finiteness of $G_{\0}$-orbits on ${\mathcal N}$. 

\begin{theorem} \label{T:finiteness N}  Let ${\mathfrak g}$ be a classical simple Lie superalgebra over ${\mathbb C}$. Then 
${\mathcal N}$ has finitely many $G_{\0}$-orbits. 
\end{theorem} 

\begin{proof} For classical simple Lie superalgebras ${\mathfrak g}$ other than $D(2,1,\alpha)$, $G(3)$ and $F(4)$, there exists an embedding 
${\mathfrak g}\hookrightarrow {\mathfrak g}^{\prime}\cong \mathfrak{gl}(m|n)$ and a supersubspace $M \subseteq {\mathfrak g}^{\prime}$ such that ${\mathfrak g}^{\prime}= M \oplus \mathfrak{g} $ and $[\mathfrak{g}, M] \subseteq M $. The embeddings and complements are described in Section~\ref{SS:embeddings} and ~\ref{S:complements}. 

Next we need to show that ${\mathcal N}_{\mathfrak g}\subseteq {\mathcal N}_{{\mathfrak g}^{\prime}}$. We will prove a stronger statement that ${\mathcal N}_{{\mathfrak g}^{\prime}}\cap {\mathfrak g}_{\1}= {\mathcal N}_{\mathfrak g}$. 
First we prove that ${\mathcal N}_{\mathfrak g}\subseteq {\mathcal N}_{{\mathfrak g}^{\prime}}\cap {\mathfrak g}_{\1}$. Let 
$z\in {\mathcal N}_{\mathfrak g}$, then $f(z)=0$ for every $f(x)\in S^{\bullet}({\mathfrak g}_{\1}^{*})^{G_{\0}}$.  
We have the identifications:
\begin{equation}
S^{\bullet}(({\mathfrak g}_{\1}^{\prime})^{*})^{G_{\0}^{\prime}}\subseteq S^{\bullet}(({\mathfrak g}_{\1}^{\prime})^{*})^{G_{\0}}=[S^{\bullet}({\mathfrak g}_{\1}^{*})\otimes S^{\bullet}(M_{\1}^{*})]^{G_{\0}}.
\end{equation} 
Under this identification, one can regard $g(x)=g(p,q)\in S^{\bullet}(({\mathfrak g}_{\1}^{\prime})^{*})^{G_{\0}^{\prime}}$. If $z\in {\mathcal N}_{\mathfrak g}\subseteq {\mathfrak g}_{\1}$ then 
$g(z)=g(z,0)=0$, thus $z\in {\mathcal N}_{{\mathfrak g}^{\prime}}$. 

The other inclusion, ${\mathcal N}_{{\mathfrak g}^{\prime}}\cap {\mathfrak g}_{\1}\subseteq {\mathcal N}_{\mathfrak g}$ uses properties of the embedding described in the first paragraph. One has 
$S^{\bullet}({\mathfrak g}_{\1}^{*})^{G_{\0}}=S^{\bullet}({\mathfrak g}_{\1}^{*})^{{\mathfrak g}_{\0}}$, and it will be more convenient to use 
Lie algebra invariants. We have 
\begin{equation}\label{E:includepoly1}
S^{\bullet}(({\mathfrak g}_{\1}^{\prime})^{*})^{{\mathfrak g}_{\0}^{\prime}}=[S^{\bullet}({\mathfrak g}_{\1}^{*})\otimes S^{\bullet}(M_{\1}^{*})]^{{\mathfrak g}_{\0}^{\prime}}\subseteq [S^{\bullet}({\mathfrak g}_{\1}^{*})\otimes S^{\bullet}(M_{\1}^{*})]^{{\mathfrak g}_{\0}}, 
\end{equation} 
and 
\begin{equation} \label{E:includepoly2}
S^{\bullet}({\mathfrak g}_{\1}^{*})^{{\mathfrak g}_{\0}}\subseteq [S^{\bullet}({\mathfrak g}_{\1}^{*})\otimes S^{\bullet}(M_{\1}^{*})]^{{\mathfrak g}_{\0}}.
\end{equation} 
The inclusion (\ref{E:includepoly2}) is given by $f(x)\mapsto f(x)\otimes 1$. Let $h(x)\in S^{\bullet}({\mathfrak g}_{\1}^{*})^{{\mathfrak g}_{\0}}$. 
Let $p=g_{\0}+m_{\0}\in {\mathfrak g}^{\prime}_{\0}$ where $g_{\0}\in {\mathfrak g}_{\0}$ and 
$m_{\0}\in M_{\0}$. Using the inclusion in (\ref{E:includepoly2}) and the fact that $[M_{\0},{\mathfrak g}_{\1}]\subseteq M_{\1}$, it follows that  
$$p.h(x)=-h([g_{\0},x]+[m_{\0},x])=-h([g_{\0},x])=g_{\0}.h(x)=0.$$
From (\ref{E:includepoly1}), if $z\in {\mathcal N}_{{\mathfrak g}^{\prime}}\cap {\mathfrak g}_{\1}$ then $z\in {\mathcal N}_{\mathfrak g}$. 

We can now prove the finiteness of $G_{\0}$-orbits on ${\mathcal N}:={\mathcal N}_{\mathfrak g}$. Let $G_{\0}\cdot y \in {\mathcal N}$. Set 
$G_{\0}^{\prime}=GL_{m}({\mathbb C})\times GL_{n}({\mathbb C})$. Then $G_{\0}^{\prime}\cdot y$ contains $G_{\0}\cdot y$, and $y\in {\mathcal N}_{{\mathfrak g}^{\prime}}$. 
Now the finiteness of $G_{\0}$-orbits on ${\mathcal N}$ follows from the finiteness of orbits for the nilpotent cone of $\mathfrak{gl}(m|n)$ and the fact that 
the intersection of any orbit in ${\mathcal N}_{{\mathfrak g}^{\prime}}$ with ${\mathfrak g}$ contains only finitely many $G_{\0}$-orbits (see Theorem~\ref{T:superrich}). 

Next we consider the remaining cases when ${\mathfrak g}$ is an exceptional Lie superalgebra. Let ${\mathfrak g}=D(2,1,\alpha)$. Then $G_{\0}\cong {SL}_{2}\times {SL}_{2}\times {SL}_{2}$ 
with ${\mathfrak g}_{\1}=V\boxtimes V \boxtimes V$ where $V$ is the two-dimensional natural representation. When $\alpha=1$, one has 
$D(2,1,\alpha)=\mathfrak{osp}(4,2)$ \cite[1.1.5]{CW}, so in this case ${\mathcal N}$ has finitely many $G_{\0}$-orbits from the argument in the preceding paragraph. 
Now the action of $G_{\0}$ on ${\mathfrak g}_{\1}$ for $D(2,1,\alpha)$ does not depend on $\alpha$. Hence, for arbitrary $\alpha$,  ${\mathcal N}$ has finitely many orbits. 

For ${\mathfrak g}=G(3)$ or $F(4)$, one can argue the finiteness as follows. Let ${\mathfrak g}=G(3)$. In this case ${\mathfrak g}_{\1}=V \boxtimes Z$ where $V$ is the 
$2$-dimensional natural representation for $SL_{2}:=SL_{2}({\mathbb C})$ and 
$Z$ is the $7$-dimensional irreducible representation for $G_2$. Let $v_{H}=(1,0)^{T}$ and $v_{L}=(0,1)^{T}$ be vectors forming the standard basis for $V$ and let $z_H$ be a highest 
weight vector for $Z$.  If $x\in {\mathcal N}$ then $x=v_{H}\otimes p_{1}+v_{L}\otimes p_{2}$. If $p_{1}=0$ or $p_{2}=0$ then we can use the fact that 
and $SL_{2}$ acts transitively on $V$ and $G_2$ acts transitively on $Z$ to show that $x$ is $G_{\0}$-conjugate to $v_{H}\otimes z_{H}$. 

Now suppose that $p_{1}\neq 0$ and $p_{2}\neq 0$. First, we can conjugate $x$ to $x_{1}=v_{H}\otimes z_H+v_{L}\otimes p_{2}^{\prime}$. 
Using the Bruhat decomposition for $G_{2}$ one can show that if $B$ is a Borel subgroup (corresponding to the positive roots) for $G_{2}$ then 
there are finitely many $B$-orbits on $Z$ with orbit representatives given by weight vectors of the form $z_{\gamma}$ where $\gamma$ is a short root for 
$G_{2}$. The group $B=T\ltimes U$ where $U$ acts trivially on $z_{H}$ and $T$ acts by scaling $z_{H}$. Thus, $x_{1}$ is $G_{\0}$-conjugate to 
$x_{2}=\alpha(v_{H}\otimes z_{H})+ v_{L}\otimes z_{\gamma}$ $\alpha\neq 0$. Moreover, since $x_{2}\in {\mathcal N}$ and satisfies a $4$th degree 
$T_{\0}$-invariant polynomial ($T_{\0}$ a maximal torus for $G_{\0}$),  it follows that $z_{\gamma}$ is not a multiple of a lowest weight vector $z_{L}$. 
Now one can use $T_{\0}$ to conjugate $x_{2}$ to $x_{3}=v_{H}\otimes z_{H}+ v_{L}\otimes z_{\gamma}$. Consequently, there are only finitely many 
$G_{\0}$-orbits on ${\mathcal N}$. 

A similar argument can be used to prove the finiteness of $G_{\0}$-orbits for ${\mathfrak g}=F(4)$. Our conclusions on the finiteness for $G(3)$ and $F(4)$ can also be 
found in  \cite[Table IV] {Kac}.
\end{proof}


\section{Connections with the Duflo-Serganova Self-Commuting Variety}\label{S:connectionswithX}

\subsection{} Let ${\mathfrak g}={\mathfrak g}_{\0}\oplus {\mathfrak g}_{\1}$ be a finite-dimensional complex Lie superalgebra with 
$\text{Lie } G_{\0}={\mathfrak g}_{\0}$. Duflo and Serganova defined the {\em self-commuting variety} as 
$${\mathcal X}=\{x\in {\mathfrak g}_{\1}:\ [x,x]=0\}.$$
The variety ${\mathcal X}$ is a $G_{\0}$-invariant conical variety of ${\mathfrak g}_{\1}$. In \cite{DS}, it was shown 
for a finite-dimensional ${\mathfrak g}$-module, $M$, one can define a subvariety ${\mathcal X}_{M}$ of ${\mathcal X}$. The 
collection of these associated varieties govern the representation theory of ${\mathfrak g}$.

\subsection{} The theorem below shows that under suitable conditions on ${\mathfrak g}$, the self-commuting variety is contained in 
the nilpotent cone of ${\mathfrak g}$. 

\begin{theorem} \label{T:inclusion} Let ${\mathfrak g}$ be a classical Lie superalgebra such 
\begin{itemize} 
\item[(a)] there exists an embedding ${\mathfrak g}\hookrightarrow {\mathfrak g}^{\prime}\cong \mathfrak{gl}(m|n)$, 
\item[(b)] there exists a supersubspace $M \subseteq {\mathfrak g}^{\prime}$ such that ${\mathfrak g}^{\prime}= M \oplus \mathfrak{g} $ and $[\mathfrak{g}, M] \subseteq M $. 
\end{itemize} 
Then ${\mathcal X}\subseteq {\mathcal N}$.  
\end{theorem} 

\begin{proof} Let ${\mathcal X}={\mathcal X}_{\mathfrak g}$ (resp. ${\mathcal X}_{{\mathfrak g}^{\prime}}$) be the self-commuting variety of ${\mathfrak g}$ 
(resp. ${\mathfrak g}^{\prime})$. Similarly, denote the nilpotent cone of ${\mathfrak g}$ (resp. ${\mathfrak g}^{\prime}$) by 
${\mathcal N}={\mathcal N}_{\mathfrak g}$ (resp. ${\mathcal N}_{{\mathfrak g}^{\prime}}$). 

Recall from Section~\ref{SS:actions} that ${\mathcal N}_{{\mathfrak g}^{\prime}}$ is defined as the zero set of  $\text{Tr}((X^+X^-)^k)$ $k=1,\ldots, l$ where $l=\text{min}\{m,n\}$. This characterization can be used to 
show that 
\begin{equation} 
{\mathcal X}_{{\mathfrak g}^{\prime}}\subseteq {\mathcal N}_{{\mathfrak g}^{\prime}}.
\end{equation}  
Moreover, using the definition of the self-commuting variety, one has 
\begin{equation} 
{\mathcal X}_{{\mathfrak g}}\subseteq {\mathcal X}_{{\mathfrak g}^{\prime}}.
\end{equation} 
Now from the proof of Theorem~\ref{T:finiteness N}, ${\mathcal N}_{{\mathfrak g}^{\prime}}\cap {\mathfrak g}_{\1}\subseteq {\mathcal N}_{\mathfrak g}$. 
Consequently, ${\mathcal X}_{\mathfrak g}\subseteq {\mathcal N}_{\mathfrak g}$.  

\end{proof}

\subsection{} We can now state and prove generalizations of the finiteness of $G_{\0}$-orbits on ${\mathcal X}$ due to Duflo and Serganova 
(cf. \cite[Theorem 4.2]{DS}). Note that their work is stated under the assumption that ${\mathfrak g}$ is a contragredient Lie superalgebra with indecomposable Cartan matrix. 

\begin{corollary}\label{c:finitenessX} Let ${\mathfrak g}$ be a classical simple Lie superalgebra over ${\mathbb C}$. Then 
\begin{itemize} 
\item[(a)] ${\mathcal X}\subseteq {\mathcal N}$,
\item[(b)] ${\mathcal X}$ has finitely many $G_{\0}$-orbits. 
\end{itemize}
\end{corollary} 

\begin{proof} We handle the case first when ${\mathfrak g}$ is not isomorphic to $D(2,1,\alpha)$, $F(4)$ or $G(3)$. In this situation, 
Theorem~\ref{T:inclusion} applies. Therefore, ${\mathcal X}\subseteq {\mathcal N}$ and ${\mathcal X}$ has finitely many $G_{\0}$-orbits. 

Now consider the case when ${\mathfrak g}=D(2,1,\alpha)$, $F(4)$ or $G(3)$. One can obtain the inclusion ${\mathcal X}\subseteq {\mathcal N}$ 
because ${\mathcal X}$ is the closure of $G_{\0}\cdot v_{H}$ where $v_{H}$ is the highest weight vector (cf. \cite[pf. of Theorem 4.2]{DS}). Since $v_{H}$ satisfies the defining equation 
for ${\mathcal N}$, one obtains the inclusion. The finiteness result for ${\mathcal X}$ for the exceptional Lie superalgebras follows from the finiteness results 
in Theorem~\ref{T:finiteness N}. 

\end{proof} 

\subsection{} For $\mathfrak{gl}(m|n)$,  we can use the parametrization of $G_{\0}$-orbit representatives for ${\mathcal N}$ to recover the Duflo-Seganova 
parametrization of $G_{\0}$-orbit representatives for ${\mathcal X}$ (cf. \cite[Theorem 4.2]{DS}).  

Let $Y$ be an orbit representative as described in Theorem~\ref{t:glmn case}(b). Then $Y\in {\mathcal X}$ if and only if $[Y,Y]=2Y^{2}=0$. A direct calculation 
shows that $Y^{2}=0$ if and only if $Y^{-}Y^{+}=0$ and $Y^{+}Y^{-}=0$. This is equivalent to the Jordan blocks $J_{i}=0$ for $i=1,2,\dots,t$, $C_{r_1}=0$, and 
$R_{r_2}=0$. Hence, $Y\in {\mathcal X}$ if and only if 
$$
Y^+ = \left[
\begin{array}{c|c}
I_r  & 0 \\ \hline
 0 & 0
\end{array}\right] 
\qquad \text{and}  \qquad 
Y^{-}=
\left[
\begin{array}{c |c | c }
 0 & 0      & 0   \\ \hline
0  & I_{s} & 0 \\ \hline
0  & 0       & 0 
\end{array}\right]
$$
This corresponds to taking a representative of a subset of linearly independent set of mutually orthogonal isotropic odd roots under the 
action of the Weyl group for $G_{\0}$ (see the paragraph after \cite[Theorem 4.2]{DS}) which is precisely how Duflo and Serganova describe their 
orbit representatives for ${\mathcal X}$. 


\section{Appendix:  Construction of Complements $M$} \label{S:complements}
 
 \subsection{} 
 For each classical Lie superalgebra $\mathfrak{g}$, an explicit matrix realization of $\mathfrak{g}$ is well known (for example, see \cite{Kac}).  We construct a matrix realization for the complement $M$ in Table \ref{table:1} below.

\begin{table}[h!]
\centering
\caption{Block matrix realization of $M$ for classical Lie superalgebras}
\label{table:1}
\begin{tabular}{c c c} 

 $\mathfrak{g}$ & $M$ &   \\ [0.5ex] 
 \hline 
 $\mathfrak{sl}(m | n)$ & $ \left[
\begin{array}{c c}
  kI_m & 0 \\ 
  0 & -kI_n
\end{array}\right] ,$ & $I_m,I_n$ identity matrices, $k \in \C$  \\ 
 $\mathfrak{osp}(2m+1 | 2n)$ & $\left[
\begin{array}{c c c|c c}
  \delta & u^t  & v^t & x & x_1 \\ 
 v & a & b & y & y_1\\ 
 u & c & a^t & z & z_1 \\ \hline
  x_1^t & z_1^t & y_1^t & d & e \\ 
  -x^t & -z^t & -y^t & f & d^t
\end{array}\right] $, & $b,c$ symmetric, $e,f$ skew-symmetric \\ 
$\mathfrak{osp}(2m | 2n)$ & $\left[
\begin{array}{c c|c c}
 a & b & y & y_1\\ 
 c & a^t & z & z_1 \\ \hline
  z_1^t & y_1^t & d & e \\ 
  -z^t & -y^t & f & d^t
\end{array}\right], $ & $b,c$ symmetric, $e,f$ skew-symmetric   \\
 $\mathfrak{q}(n)$ & $ \left[
\begin{array}{c c}
  a & b \\ 
  -b & -a
\end{array}\right]  $ &   \\
 $\mathfrak{p}(n)$ & $\left[
\begin{array}{c c}
  a & b \\ 
  c & a^t
\end{array}\right]  $, & $b$ skew-symmetric, $c$ symmetric\\ [1ex] 
 \hline
\end{tabular}

\end{table}

\subsection{} We now check that each of the non-exceptional classical Lie superalgebras $\mathfrak{g}$ except $\mathfrak{gl}(m\lvert n)$ satisfy the hypotheses of Theorem \ref{T:superrich} case-by-case. From the construction of $M_{\1}$ in each case below it follows that $\mathfrak{gl}(m | n)_{\1}= \mathfrak{g}_{\1} \oplus M_{\1}$, since each generator $E_{ij}$ of $\mathfrak{gl}(m | n)_{\1}$ can be written as a sum of a an element of $\mathfrak{g}_{\1}$ and an element of $M_{\1}$ in an obvious way.  Direct calculation shows that $[\mathfrak{g}_{\bar{i}}, M_{\bar{j}} ]\subseteq M_{\overline{i+j}}$ in each case.  Sample calculations are given below for each classical Lie superalgebra when $i=j=1$.  In each case, let $X \in \mathfrak{g}_{\1}$ and $Y \in M_{\1}$ so that $[X,Y] = XY+YX.$ 

\smallskip

\begin{itemize}
    \item $\mathfrak{sl}(m | n):$ $Y = 0,$ so $[X,Y] = 0 $.

    \item $\mathfrak{q}(n): $ $ X =  \left[
\begin{array}{c c}
  0 & b \\ 
  b & 0
\end{array}\right], Y =  \left[
\begin{array}{c c}
  0 & d \\ 
  -d & 0
\end{array}\right]  $.  Then

\[ [X,Y] = \left[
\begin{array}{c c}
  db-bd & 0 \\ 
  0 & -(db-bd)
\end{array}\right]. \]
        
    \item $\mathfrak{p}(n): $ $ X =  \left[
\begin{array}{c c}
  0 & b \\ 
  c & 0
\end{array}\right], Y =  \left[
\begin{array}{c c}
  0 & a \\ 
  d & 0
\end{array}\right]  $ with $b,d$ symmetric and $a,c$ skew-symmetric.  Then

\[ [X,Y] = \left[
\begin{array}{c c}
  bd+ac & 0 \\ 
  0 & ca+db
\end{array}\right] = \left[
\begin{array}{c c}
  bd+ac & 0 \\ 
  0 & (bd+ac)^t
\end{array}\right]. \]

    \item $\mathfrak{osp}(2m +1 | 2n):$ $X = \left[
\begin{array}{c c c|c c}
   &   &  & x & x_1 \\ 
  &  &  & y & y_1\\ 
  &  &  & z & z_1 \\ \hline
  x_1^t & z_1^t & y_1^t &  &  \\ 
  -x^t & -z^t & -y^t &  & 
\end{array}\right], Y =  \left[
\begin{array}{c c c|c c}
  &   &  & a & a_1 \\ 
  &  &  & b & b_1\\ 
  &  &  & c & c_1 \\ \hline
  a_1^t & c_1^t & b_1^t &  &  \\ 
  -a^t & -c^t & -b^t &  & 
\end{array}\right]. $ Then $[X,Y] =  \left[
\begin{array}{c c}
  A & 0 \\ 
  0 & B
\end{array}\right] $   where the matrices $A,B$ have block forms 
    
    \[ A =  \left[
\begin{array}{c c c}
  xa_1^t-x_1a^t-ax_1^t+a_1x^t & xc_1^t-x_1c^t-az_1^t+a_1z^t  & xb_1^t-x_1b^t-ay_1^t+a_1y^t \\ 
  ya_1^t-y_1a^t-bx_1^t+b_1x^t & yc_1^t-y_1c^t-bz_1^t+b_1z^t  & yb_1^t-y_1b^t-by_1^t+b_1y^t  \\ 
  za_1^t-z_1a^t-cx_1^t+c_1x^t & zc_1^t-z_1c^t-cz_1^t+c_1z^t  & zb_1^t-z_1b^t-cy_1^t+c_1y^t  
\end{array}\right] \]
    
    \[ B = \left[
\begin{array}{c c}
  -x_1^ta-z_1^tb-y_1^tc+a_1^tx+c_1^ty+b_1^tz & -x_1^ta_1-z_1^tb_1-y_1^tc_1+a_1^tx_1+c_1^ty_1+b_1^tz_1   \\ 
  x^ta+z^tb+y^tc-a^tx-c^ty-b^tz & x^ta_1+z^tb_1+y^tc_1-a^tx_1-c^ty_1-b^tz_1   
\end{array}\right] \]
with the blocks satisfying the relations 

\[ A_{12} = A_{31}^t, A_{13} = A_{21}^t, A_{33} = A_{22}^t, B_{22} = B_{11}^t \]  

and with $A_{23}, A_{32}$ symmetric and $B_{12}, B_{21}$ skew-symmetric.  \\
    \smallskip

\item$\mathfrak{osp}(2m | 2n)$:  Since this superalgebra is obtained by deleting the first row and first column of $\mathfrak{osp}(2m+1 | 2n)$, the calculations in this case are obtained in a similar manner.

\end{itemize}

\end{document}